\documentclass[10pt]{article}
\usepackage[left=1in,top=1in,right=1in,bottom=1in,letterpaper]{geometry}
\usepackage{subfigure}

\usepackage{amsmath,amssymb,amsthm}
\usepackage{latexsym,amsfonts,amscd,amsxtra,amstext}
\usepackage{algorithm,algorithmic}
\usepackage{url}

\usepackage{epsfig} 
\usepackage{threeparttable}
\usepackage{index}
\usepackage{subfigure}

\newcommand{\RR}{\mathbb{R}}

\newcommand{\sign}{\mathrm{sign}}

\newcommand{\prox}{{\mathbf{prox} }} 
\newcommand{\shrink}{{\mathbf{shrink} }} 

\DeclareMathOperator*{\argmin}{arg\,min}

\DeclareMathOperator*{\Min}{minimize}
\DeclareMathOperator*{\Max}{maximize}

\newcommand{\st}{{\quad\text{s.t.}~}}

\newcommand{\bc}{\begin{center}}
\newcommand{\ec}{\end{center}}

\newcommand{\bdm}{\begin{displaymath}}
\newcommand{\edm}{\end{displaymath}}

\newcommand{\beq}{\begin{equation}}
\newcommand{\eeq}{\end{equation}}

\newcommand{\bfl}{\begin{flushleft}}
\newcommand{\efl}{\end{flushleft}}

\newcommand{\bt}{\begin{tabbing}}
\newcommand{\et}{\end{tabbing}}

\newcommand{\beqn}{\begin{eqnarray}}
\newcommand{\eeqn}{\end{eqnarray}}

\newcommand{\beqs}{\begin{align*}} 
\newcommand{\eeqs}{\end{align*}}  

\newtheorem{theorem}{Theorem}

\newtheorem{definition}{Definition}
\newtheorem{corollary}{Corollary}

\newtheorem{lemma}{Lemma}

\begin{document}

\title{Projected shrinkage algorithm for box-constrained $\ell_1$-minimization}

\author{Hui Zhang  \and LiZhi Cheng \thanks{
College of Science, National University of Defense Technology,
Changsha, Hunan,  China, 410073. Emails: \texttt{h.zhang1984@163.com}( Hui Zhang); \texttt{clzcheng@nudt.edu.cn}(Lizhi Cheng).}
}
\date{\today}

\maketitle

\begin{abstract}
 Box-constrained $\ell_1$-minimization can perform remarkably better than classical $\ell_1$-minimization when correction box constraints are available. And also many practical $\ell_1$-minimization models indeed involve box constraints because they take certain values from some interval. In this paper, we propose an efficient iteration scheme, namely projected shrinkage (ProShrink) algorithm, to solve a class of box-constrained $\ell_1$-minimization problems. A key contribution in our technique is that a complicated proximal point operator appeared in the deduction can be equivalently simplified into a projected shrinkage operator. Theoretically, we prove that ProShrink enjoys a convergence of both the primal and dual point sequences. On the numerical level, we demonstrate the benefit of adding box constraints via sparse recovery experiments.
\end{abstract}


\section{Introduction}
The past two decades has witnessed the wide application of $\ell_1$-minimization models in signal and image processing, compressive sensing, machine learning, statistic, and more. The success of $\ell_1$ minimization is mainly due to that the $\ell_1$-norm can well reflect sparse prior. Recently, it was observed that other auxiliary information of sparse solutions, such as partial support set \cite{VL} and nonnegative sparsity \cite{b,d}, could help fit practical models.
In this paper, instead of studying the theoretical benefit of modeling auxiliary priors, we are interested in designing efficient algorithms to solve the $\ell_1$-minimization problems with auxiliary box constraints:
\begin{equation}\label{BP}
\Min_x\|x\|_1,\quad\st~A x=b, x\in \mathcal{X}
\end{equation}
and
\begin{equation}\label{Aug}
\Min_x\|x\|_1+\frac{1}{2\tau}\|x\|_2^2,\quad\st~A x=b, x\in \mathcal{X}
\end{equation}
where  $A\in \RR^{n\times m}, b\in\RR^m$ are given, $\tau$ is an augmented parameter, and $\mathcal{X}$ is some box-constrained set. The above problems are obtained separately by imposing box constraints to the basis pursuit model \cite{cd}:
\begin{equation}\label{BPo}
\Min_x\|x\|_1,\quad\st~A x=b
\end{equation}
and the augmented $\ell_1$ norm model \cite{ly}:
\begin{equation}\label{Augo}
\Min_x\|x\|_1+\frac{1}{2\tau}\|x\|_2^2,\quad\st~A x=b,
\end{equation}
both of which have been proved powerful for sparse recovery. Adding box constraints to classical $\ell_1$ minimization on one hand extends the range of models \eqref{BPo} and \eqref{Augo} to include more practical models in application, and on the other hand can help improve the ability of sparse recovery of them when correct box constraints are available; the second point of view shall be demonstrated numerically later on.  Similar benefit of adding box constraints to classical matrix completion has been observed in a recent paper \cite{tmr} which was posted on arXiv at the time of the writing of the present paper.

Due to the existing of the strongly convex term $\frac{1}{\tau}\|x\|_2^2$, it has been explained in several papers \cite{z4,z5,ly} that models \eqref{Aug} and \eqref{Augo} have computational advantages over their correspondences \eqref{BP} and \eqref{BPo}. Besides, applying the proximal point algorithm \cite{pb} to models \eqref{BP} or \eqref{BPo}  generates a series of subproblems similar to \eqref{Aug} or \eqref{Augo}. Therefore, the center assignment of solving problems \eqref{BP}-\eqref{Augo} reduces to studying the following generalized problem
\begin{equation}\label{cent}
\Min_x\|x\|_1+\frac{1}{2\tau}\|x-u\|_2^2,\quad\st~A x=b, x\in \mathcal{X}
\end{equation}
where $u$ is a given vector.
With the help of the Lagrange dual analysis and by noticing the strong convexity of the objective function, in this study we derive a projected shrinkage (ProShink) algorithm for solving \eqref{cent}. By the Nesterov techniques \cite{n1}, the proposed algorithm can be speeded up; we present an accelerated scheme as well.  Theoretically, we prove the convergence of both the primal and dual point sequences of ProShink.  A key contribution in our technique is that a complicated proximal operator appeared in the deduction can be equivalently simplified into a projected shrinkage operator. This can also be applied to simplifying standard forward-backward splitting algorithm for the boxed-constrained basis pursuit denoising problem:
\begin{equation}\label{PFB}
\Min_{x\in\mathcal{X}}\|x\|_1+\frac{1}{2\lambda}\|Ax-b\|_2^2,
\end{equation}
where $\lambda$ is a positive paramter.

The rest of paper is organized as follows. In section 2, we introduce some basis concepts of constrained convex optimization and obtain important properties about the shrinkage operator. In section 3, under the Lagrange dual analysis, we propose the ProShrink algorithm and prove its convergence, and meanwhile we present detailed iteration schemes for solving models \eqref{BP}, \eqref{Aug}, and \eqref{PFB}.  In section 4, we do sparse recovery experiments to demonstrate the benefit of adding box constrains.

\section{Notation and important properties}
In this paper, we restrict our attention onto two classes of intervals.
The first class is:
 \begin{subequations}
\begin{align*}
T_1=\{I: &  I=[c, \infty), c>0, ~~\textrm{or}~~I=(-\infty, c], c<0, ~~\textrm{or} \\
& I=[c, d], 0<c<d ~~\textrm{or}~~I=[d, c], d<c<0\};
\end{align*}
\end{subequations}
The second class is
$$T_2=\{I: I=[c, d], c<0<d\}.$$
The box constraint $\mathcal{X}$ appeared in all models mentioned before is defined as $\mathcal{X}=I_1\times I_2\times \cdots \times I_n$, where $I_i\in T_1\bigcup I_2$.  Throughout this paper, we assume that $\mathcal{X}\bigcap \{x: Ax=b\}\neq \emptyset$.

\subsection{Basic concepts and properties}
First, we introduce proximal point operator and its important properties.
\begin{definition}
Let $f: \RR^n\rightarrow R\cup\{+\infty\}$ be a closed proper convex function. The proximal operator \cite{m}   $\prox_f :\RR^n\rightarrow \RR^n$ is defined by
\begin{equation}
\prox_f(v)=\argmin_x\left(f(x)+\frac{1}{2}\|x-v\|_2^2\right).
\end{equation}
\end{definition}
Since the objective function is strongly convex and proper, $\prox_f(v)$ is properly defined for every $v\in\RR^n$.  The following properties \cite{m,pb} will be used in our analysis.
 \begin{lemma}\label{lem1}
 Let $f: \RR^n\rightarrow R\cup\{+\infty\}$ be a closed proper convex function. Then, for all $x, y\in \RR^n$ the proximal operator $\prox_{f(\cdot)}$ satisfies the followings:
 \begin{enumerate}
  \item Firmly nonexpansive: $$ \|\prox_{f(\cdot)}(x)-\prox_{f(\cdot)}(y)\|_2^2\leq \langle x-y, \prox_{f(\cdot)}(x)-\prox_{f(\cdot)}(y)\rangle$$

  \item Lipschitz continuous: $\|\prox_{f(\cdot)}(x)-\prox_{f(\cdot)}(y)\|_2\leq \|x-y\|_2$.
  \end{enumerate}
 \end{lemma}
 If $f$ is fully separable, meaning that $f(x) =\sum_{i=1}^n f_i(x_i)$, then
 \begin{equation} \label{fulls}
(\prox_{f(\cdot)}(x))_i=\prox_{f_i(\cdot)}(x_i).
\end{equation}

Second, we need to introduce convex projected operator and projected subgradient to deal with box constraints.
\begin{definition}[convex projected operator]
$[x]_\mathcal{X}^+:=\arg\min_{y\in\mathcal{X}}\|x-y\|$.
\end{definition}

The following property of  projected operator shall be often encountered in our deduction.
\begin{lemma}\label{projp}
For any interval $I\in T_1\bigcup T_2$, we have that
$[\tau\cdot w]_I^+=\tau\cdot[w]^+_{I/\tau}$
holds for arbitrary $w\in\RR$ and positive parameter $\tau$.
\end{lemma}
\begin{proof}
We begin with the definition of projected operator and derive that
\begin{subequations}
\begin{align*}
[\tau\cdot w]_I^+  =\arg\min_{y\in I}\|y-\tau w\| =\arg\min_{y\in I}\|\tau^{-1}y- w\|
 =\tau\cdot\arg\min_{z\in I/\tau}\|z- w\|  =\tau\cdot[w]^+_{I/\tau}.
\end{align*}
\end{subequations}
This completes the proof.
\end{proof}

\begin{definition}[projected subgradient]Define
$\partial_{\mathcal{X}}^+ f(x):=\{g:  g=x-[x-h]_\mathcal{X}^+, h\in \partial f(x)\}$
Without confusion, we also denote $\partial_{\mathcal{X}}^+ f(x)$ by $x-[x-\partial f(x)]_\mathcal{X}^+ $.
\end{definition}

With projected subgradient, we can state a necessary and sufficient condition which guarantees a vector to be a minimizer to a class of constrained convex optimization problems.
\begin{lemma}\label{lem:optcond}
Let $f(x)$ be proper convex and $\mathcal{X}$ nonempty, closed, and convex. Then, we have
$$x^*\in \arg\min_{x\in\mathcal{X}} f(x)\Leftrightarrow 0\in \partial_{\mathcal{X}}^+ f(x^*)$$
\end{lemma}
\begin{proof}
The following two facts will be used in our deduction:

\textbf{Fact 1.} $z\in[x]^+_\mathcal{X} \Leftrightarrow  \langle z-x, z-y\rangle\leq 0, \forall y\in \mathcal {X}$;

\textbf{Fact 2.} $\forall y\in \mathcal{X}, \exists h\in\partial f(x^*), \langle h, y-x^*\rangle \geq 0\Leftrightarrow\ x^*\in\arg\min_{x\in\mathcal{X}}f(x)$.

With these two facts, we derive that
 \begin{subequations}
\begin{align}
0\in \partial_{\mathcal{X},\tau}^+ f(x^*) &\Leftrightarrow   x^*\in [x^*-\partial f(x^*)]_\mathcal{X}^+  \\
&\Leftrightarrow   \exists h\in\partial f(x^*), \textrm{such that}~~x^*=[x^*- h]_\mathcal{X}^+ \\
&\Leftrightarrow   \langle x^*-(x^*- h), x^*-y\rangle \geq 0, \forall y\in\mathcal{X} \\
&\Leftrightarrow   \langle h, y-x^*\rangle \geq 0, \exists h\in \partial f(x^*), \forall y\in\mathcal{X}\\
&\Leftrightarrow    x^*\in\arg\min_{x\in\mathcal{X}}f(x).
\end{align}
\end{subequations}
This completes the proof.
\end{proof}

\subsection{Projected shrinkage operator}
In this part, we build an important formulation that links the proximal point operator and the projected shrinkage operation. In order to  establish that formulation, we need two lemmas.
\begin{lemma}\label{shrinkproj}
Let $\shrink$ be the shrinkage operator defined by $\shrink(s)=\sign(s)\max\{|s|-1,0\}$ and let $I\in T_1$. Then, we always have that
\begin{equation}\label{shrinksign}
[\shrink(q)]_I^+=[q-\sign(c)]_I^+
\end{equation}
holds for arbitrary  $q \in \RR$, where $c$ appears in the definition of $I$.
\end{lemma}
\begin{proof}
Recall that $I=[c, \infty)$ with $c>0$, or $I=(-\infty, c]$ with $c<0$, or $
 I=[c, d]$ with $0<c<d $,  or $I=[d, c],$ with $d<c<0$. So it is easy to observe the following fact:

If $c>0$, then $[q+1]_I^+\equiv c$ for each $q<-1$ and $[0]_I^+\equiv c$;

If $c<0$, then $[q-1]_I^+\equiv c$ for each $q>1$ and $[0]_I^+\equiv c$.

Thus, together with the definition of the shrinkage operator, for $c>0$ we have that
\begin{eqnarray}
[\shrink(q)]_I^+
&=&\left\{\begin{array}{lc}
[q-1]_I^+, & q>1 \\
{[0]_I^+}, & -1\leq q\leq 1 \\
{[q+1]_I^+}, & q<-1
\end{array}  \right. \\
&=&\left\{\begin{array}{lc}
[q-1]_I^+, & q>1 \\
{c}, & q\leq 1
\end{array}  \right.=[q-1]_I^+
\end{eqnarray}

and for $c<0$ have that
\begin{equation}
[\shrink(q)]_I^+
=\left\{\begin{array}{lc}
[q+1]_I^+, & q<-1 \\
{c}, & q\geq -1
\end{array}\right.=[q+1]_I^+.
\end{equation}
On the other hand, $[q-\sign(c)]_I^+=[q-1]_I^+$ when $c>0$ and $[q-\sign(c)]_I^+=[q+1]_I^+$ when $c<0$. So the relationship \eqref{shrinksign} holds.
\end{proof}

\begin{lemma}\label{shrinkproj2}
Let $I\in T_1\bigcup T_2$ and $I_\tau(t)=\tau\cdot|t|+\delta_I(t)$ where $\delta_I(\cdot)$ is the indicator function. Then, we always have that
\begin{equation}\label{shrinkprox}
[\tau\cdot\shrink(\tau^{-1}q)]_I^+=\prox_{I_\tau(\cdot)}(q)
\end{equation}
holds for arbitrary  $q \in \RR$.
\end{lemma}
\begin{proof}
By the definition of proximal point operator, we derive that
 \begin{subequations}
\begin{align}
t^*:=\prox_{I_\tau(\cdot)}(q) =\arg\min_{t\in\RR} I_\tau(t)+\frac{1}{2}(t-q)^2  =\arg\min_{t\in I}\tau\cdot|t|+\frac{1}{2}(t-q)^2\label{prob1}.
\end{align}
\end{subequations}

If $I\in T_1$, then $|t|=\sign(c)\cdot t$ and hence
 \begin{subequations}
\begin{align*}
t^*=\arg\min_{t\in I}\tau\cdot|t|+\frac{1}{2}(t-q)^2=\arg\min_{t\in I}\tau\cdot\sign(c)\cdot t+\frac{1}{2}(t-q)^2.
\end{align*}
\end{subequations}
Applying Lemma \ref{lem:optcond} yields to $t^*=[q-\tau\cdot\sign(c)]_I^+$. Together with Lemmas \ref{projp} and \ref{shrinkproj}, we derive that
$$t^*=[\tau\cdot(\tau^{-1}q-\sign(c))]_I^+=\tau\cdot[\tau^{-1}q-\sign(c)]_{I/\tau}^+=\tau\cdot [\shrink(\tau^{-1}q)]_{I/\tau}^+=[\tau\cdot\shrink(\tau^{-1}q)]_I^+.$$
So relationship \eqref{shrinkprox} holds when $I\in T_1$.

If $I\in T_2$, then again invoking Lemma \ref{lem:optcond} yields to $t^*\in [q-\tau\cdot\partial |t|_{t=t^*}]_I^+$. Such $t^*$ must be the unique solution to problem \eqref{prob1} because its objective functions is strongly convex. Thus, it suffices to show that $p(q)= [\tau\cdot\shrink(\tau^{-1}q)]_I^+$ satisfies the following inclusion:
$$t\in [q-\tau\cdot \partial |t|]_I^+.$$
Now, we check $p(q)$ case-by-case:

\textbf{Case 1: $p(q)>0$}. Since interval $I=[c, d]\in T_2$ satisfies $c<0<d$, condition $p(q)>0$ implies $\shrink(\tau^{-1}q)>0$ and hence $\tau^{-1}q>1$. Then, by Lemmas \ref{projp} and \ref{shrinkproj}, and together with the definition of shrinkage operator we derive that
$$[ q-\tau\cdot \partial \|p(q)\|_1]^+_I=[q-\tau]^+_I= [\tau\cdot(\tau^{-1}q-1)]^+_I=\tau\cdot[\tau^{-1}q-1]^+_{I/\tau} = \tau\cdot[\shrink(\tau^{-1}q)]_{I/\tau}^+=[\tau\cdot\shrink(\tau^{-1}q)]_I^+=p(q).$$

\textbf{Case 2: $p(q)<0$}. Condition $p(q)<0$ implies $\tau^{-1}q < -1$. Then, similarly to the argument in Case 1, we have that
$$[ q-\tau\cdot \partial \|p(q)\|_1]^+_I=[q-\tau]^+_I= [\tau\cdot(\tau^{-1}q+1)]^+_I=\tau\cdot[\tau^{-1}q+1]^+_{I/\tau} = \tau\cdot[\shrink(\tau^{-1}q)]_{I/\tau}^+=[\tau\cdot\shrink(\tau^{-1}q)]_I^+=p(q).$$

\textbf{Case 3: $p(q)=0$}. Condition $p(q)=0$ implies $-1\leq \tau^{-1}\cdot q\leq 1$. Then, noting the fact that $[-1, 1]=\partial \|0\|_1$, we have that
$$[q- \tau\cdot\partial \|p(q)\|_1]^+_I=[ q- \tau\cdot\partial \|0\|_1]^+_I=\tau\cdot[ \tau^{-1}q- \cdot\partial \|0\|_1]^+_{I/\tau}\ni 0.$$
This completes the proof.

\end{proof}

Now, we are ready to build the most important formulation in this study.
\begin{corollary}\label{proshr}
Define the projected shrinkage operator $[\shrink(v)]_\mathcal{X}^+$ for a vector $v\in\RR^n$ via
$$\left([\shrink(v)]_\mathcal{X}^+\right)_i=[\shrink(v_i)]_{I_i}^+, i= 1, 2, \cdots, n.$$
And let $\mathcal{X}_\tau(x)=\tau\cdot\|x\|_1+\delta_\mathcal{X}(x)$.
Then, it holds
\begin{equation}\label{key}
\boxed{[\tau\cdot\shrink(\tau^{-1}v)]_\mathcal{X}^+=\prox_{\mathcal{X}_\tau(\cdot)}(v)}.
\end{equation}
\end{corollary}
\begin{proof}
Noting that $\mathcal{X}_\tau(x)=\sum_{i=1}^n
\tau|x_i|+\delta_{I_i}(x_i)$ and the property \eqref{fulls}, together with Lemma \ref{shrinkproj2}, the conclusion follows.
\end{proof}
The significance of formulation \eqref{key} is two-fold: the expression based on proximal point operator will be used for convergence analysis; whilst that expressed by the projected shrinkage operator is for computational consideration due to its simplicity.

\section{Projected shrinkage algorithm}
In this section,  we derive a Lagrange dual problem of \eqref{cent} and the ProShrink algorithm for solving it. Following the line of proof thought in paper \cite{z3}, we prove the convergence of both the primal and dual point sequences of the ProShrink algorithm.
\subsection{Lagrange dual analysis}
The Lagrangian of the augmented convex model (\ref{cent}) is
\begin{equation}
L(x, y)= \|x\|_1 +\frac{1}{2\tau} \|x-u\|_2^2 +\langle y, b-A x\rangle.
\end{equation}
The Lagrange dual function is
\begin{equation}
D(y)=\Min_{x\in\mathcal{X}}L(x,y).
\end{equation}
For any vector $y\in \RR^m$, the $x$-minimization problem above is a strongly convex program and hence has a unique solution $x^*(y)$ that satisfies
  \begin{subequations}
\begin{align}
x^*(y) &=\arg\min_{x\in \mathcal{X}} L(x, y)=\|x\|_1 +\frac{1}{2\tau} \|x-u\|_2^2 +\langle y, b-A x\rangle \\
&=\arg\min_{x\in \mathcal{X}} \|x\|_1 +\frac{1}{2\tau}\|x-u-\tau\cdot A^Ty\|_2^2\\
&=\arg\min_{x\in \mathcal{X}} \tau\cdot\|x\|_1 +\frac{1}{2}\|x-u-\tau\cdot A^Ty\|_2^2=\prox_{\mathcal{X}_\tau(\cdot)}(u+\tau\cdot A^Ty),
\end{align}
\end{subequations}
where $\mathcal{X}_\tau(x)=\tau\cdot\|x\|_1+\delta_\mathcal{X}(x)$. By formulation \eqref{key} in Corollary \ref{proshr}, we obtain
\begin{equation}\label{primaldual}
x^*(y)=\prox_{\mathcal{X}_\tau(\cdot)}(u+\tau\cdot A^Ty) = [\tau\cdot\shrink(\tau^{-1}u +A^Ty)]_\mathcal{X}^+.
\end{equation}

Now, $D(y)= L(x^*(y),y)=L([\tau\cdot\shrink(\tau^{-1}u +A^Ty)]^+_\mathcal{X}, y)$. Thus, we can write down the Lagrange dual problem of (\ref{cent}) as follows:
\begin{equation}\label{Dual}
\Max_y  L([\tau\cdot\shrink(\tau^{-1}u + A^Ty)]^+_\mathcal{X}, y).
\end{equation}
It is well known in convex analysis \cite{r} that the dual objective function $L([\tau\cdot\shrink(\tau^{-1}u + A^Ty)]^+_\mathcal{X}, y)$ is gradient-Lipschitz-continuous due to the strong convexity of the primal objective function $\|x\|_1 +\frac{1}{2\tau} \|x-u\|_2^2$. And moreover, the gradient of dual objective function is given by $$\nabla D(y) = b-Ax^*(y)=b-A[\tau\cdot\shrink(\tau^{-1}u + A^Ty)]^+_\mathcal{X}$$
Each solution to the dual problem \eqref{Dual} can generate the unique solution to the primal problem \eqref{cent} via formulation \eqref{primaldual}. This fact is stated in the following lemma.

\begin{lemma}\label{dset}
 Let $x^*$ be the unique solution to problem \eqref{cent} and  $\mathcal{X}\bigcap \{x: Ax=b\}\neq \emptyset$. Then the dual solution set to problem \eqref{Dual} is
\begin{equation}\label{dualsset}
\mathcal{Y}=\left\{y:  x^*=[\tau\cdot\shrink(\tau^{-1}u + A^Ty)]^+_\mathcal{X}\right\},
\end{equation}
  which is  nonempty and convex.
  \end{lemma}
   \begin{proof}
  By Lemma \ref{lem1}, for $\forall y, \tilde{y}\in\RR^m$ we have that
  \begin{subequations}
\begin{align*}
&\langle \nabla D(y)-\nabla D(\tilde{y}), y-\tilde{y}\rangle \\
=&\langle [\tau \cdot\shrink(\tau^{-1}u +A^Ty)]^+_{\mathcal{X}}-[\tau \cdot \shrink(\tau^{-1}u + A^T\tilde{y})]^+_{\mathcal{X}},A^Ty-A^T\tilde{y}\rangle \\
=&\tau^{-1} \langle \prox_{\mathcal{X}_\tau(\cdot)}(u+\tau\cdot A^Ty)-\prox_{\mathcal{X}_\tau(\cdot)}(u+\tau\cdot A^T\tilde{y}),\tau\cdot A^Ty-\tau\cdot A^T\tilde{y}\rangle \\
\geq& \tau^{-1}  \|\prox_{\mathcal{X}_\tau(\cdot)}(u+\tau\cdot A^Ty)-\prox_{\mathcal{X}_\tau(\cdot)}(u+\tau\cdot A^T\tilde{y})\|_2^2\geq 0,
\end{align*}
\end{subequations}
which implies that the dual objective function $D(y)$ is convex. Thus, the dual solution set is
\begin{subequations}
\begin{align}
\mathcal{Y}^{'} & = \{y: \nabla D(y)=0\} \\
& =  \{y: A [\tau \cdot \shrink(\tau^{-1}u +A^Ty)]^+_\mathcal{X}=b\},
\end{align}
\end{subequations}
which must be nonempty and convex by assumption $\mathcal{X}\bigcap \{x: Ax=b\}\neq \emptyset$ and the convexity of $D(y)$. Now, it suffices to show $\mathcal{Y}=\mathcal{Y}^{'}$. On one hand, we have $\mathcal{Y}\subseteq \mathcal{Y}^{'}$ since $Ax^*=b$. On the other hand, let $y^{'}\in \mathcal{Y}^{'}$, i.e., $y^{'}$ is some dual solution. Then, $x^{'}=[\tau\cdot\shrink(\tau^{-1}u + A^Ty)^{'}]^+_\mathcal{X}$ is a primal solution and it must equal $x^*$ by uniqueness. So $y^{'}\in\mathcal{Y}$ and hence $\mathcal{Y}^{'}\subseteq \mathcal{Y}$, which completes the proof.
\end{proof}

\subsection{Algorithm schemes}
Applying the gradient iteration to the dual objective $D(y)$ gives:
\begin{equation}\label{form1}
y^{k+1}=y^k+h(b-A[\tau \cdot \shrink(\tau^{-1}u +A^Ty^{k})]_\mathcal{X}^+),
\end{equation}
where $h>0$ is the step size whose range shall be studied later for convergence. By setting $x^{k+1}= [\tau \cdot \shrink(\tau^{-1}u +A^Ty^{k})]_\mathcal{X}^+$, we obtain the
equivalent iteration in the primal-dual form:
  \begin{eqnarray}\label{mainalg}
\left\{\begin{array}{ll}
x^{k+1}= [\tau \cdot \shrink(\tau^{-1}u +A^Ty^{k})]_\mathcal{X}^+ \\
y^{k+1}=y^{k}+h (b-Ax^{k+1}).
\end{array} \right.
\end{eqnarray}
Because the projected shrinkage operator is involved, we call \eqref{mainalg} projected shrinkage algorithm.
 Recall that the linearized Bregman (LBreg) algorithm \cite{y1,c1} has the following form:
  \begin{eqnarray}
\left\{\begin{array}{ll}
x^{k+1}= \tau \cdot \shrink (A^Ty^{k}) \\
y^{k+1}=y^{k}+h (b-Ax^{k+1}).
\end{array} \right.
\end{eqnarray}
Therefore, the ProShrink algorithm can be viewed as a generalization of the LBreg algorithm.

Applying Nesterov's accelerated scheme \cite{n1}, we obtain an accelerated ProShrink algorithm with the following form:
 \begin{eqnarray}\label{acc1}
\left\{\begin{array}{ll}
x^{k+1}= [\tau \cdot \shrink (\tau^{-1}u + A^Ty^{k})]^+_\mathcal{X}; \\
z^{k+1}=y^{(k)}+h (b-Ax^{k+1});\\
\gamma_k=(\sqrt{\theta_k+4}-\theta_k)/2;\\
\beta_{k+1} =(1-\theta_k)\gamma_k, \theta_{k+1}=\theta_k\gamma_k; \\
y^{k+1}= z^{k+1}+\beta_{k+1}(z^{k+1}-z^k).
\end{array} \right.
\end{eqnarray}
In addition, it is predictable that the adaptive restart technique developed in \cite{oc} can further accelerate the scheme \eqref{acc1}; Such acceleration for the LBreg algorithm was observed in paper \cite{z1}.

Now, let us return to models \eqref{BP} and \eqref{Aug}. Model \eqref{Aug} can be solved by ProShrink \eqref{mainalg} or its acceleration \eqref{acc1} with $u=0$. To solve model \eqref{BP}, we apply the proximal point algorithm and obtain a series of subproblems as follows:
$$z^{k+1}=\arg\min \{\|x\|_1+\frac{1}{2\lambda_k}\|x-z^k\|^2, Ax=b, x\in\mathcal{X}\}$$
where $\lambda_k$ are positive parameters. Each subproblem above can be well solved by  ProShrink \eqref{mainalg} as well. We write down the iteration scheme without detailed derivation:
  \begin{eqnarray}\label{mainalg2}
\left\{\begin{array}{ll}
x^{i+1}= [\lambda_k \cdot \shrink(\frac{1}{\lambda_k}z^k + A^Ty^{i})]_\mathcal{X}^+ \\
y^{i+1}=y^{i}+h (b-Ax^{i+1}).
\end{array} \right.
\end{eqnarray}
The subproblem can also be solved by the accelerated ProShrink scheme \eqref{acc1}.

At last, the standard forward-backward splitting algorithm for model \eqref{PFB} is
\begin{equation}
x^{k+1}=\prox_{\mathcal{X}_{\gamma_k}(\cdot)}\left(x^k-\frac{\gamma_k}{\lambda}A^T(Ax^k-b)\right),
\end{equation}
where $\gamma_k$ are the step sizes. The main difficulty of the above iteration is to compute the proximal point operator of $\mathcal{X}_{\gamma_k}(\cdot)$.
Utilizing formulation \eqref{key}, this can be overcome and the iteration can be simplified into
\begin{equation}
x^{k+1}=[\gamma_k\cdot\shrink(\gamma^{-1}_kx^k-\lambda^{-1}A^T(Ax^k-b))]^+_\mathcal{X}.
\end{equation}

\subsection{Convergence analysis}
In this part, we prove the convergence of primal sequence $\{x^k\}$ and dual sequence $\{y^k\}$ in iteration (\ref{mainalg}).
\begin{theorem}\label{thm:cvg}
Set step size $h\in (0, \frac{2}{\tau\|A\|^2})$ and $y^0=0$ in iteration (\ref{mainalg}). Let $x^*$ be the unique minimizer to problem (\ref{Aug}) and $\mathcal{Y}$ be the solution set to problem \eqref{Dual}. Then, $\lim_{k\rightarrow +\infty}x^k=x^*$, and there exists a point $\bar{y}\in\mathcal{Y}$ such that $\lim_{k\rightarrow +\infty} y^k=\bar{y}$.
\end{theorem}
This theorem can be proved in the same manner as that in paper \cite{z3}. For completeness, we provide a proof below.
\begin{proof}
Let $\hat{y}\in \mathcal{Y}$. By Lemma \ref{dset}, we have $x^*=[\tau\cdot\shrink(\tau^{-1}u +A^T\hat{y})]^+_\mathcal{X}$. Together with $x^{k+1}= [\tau \cdot \shrink(\tau^{-1}u +A^Ty^{k})]_\mathcal{X}^+$ and Lemma \ref{lem1} and Corollary \ref{proshr}, we derive
 \begin{subequations}
\begin{align}
& \langle A^Ty^k-A^T\hat{y}, x^{k+1}- x^*\rangle \\
=&\langle A^Ty^k-A^T\hat{y}, [\tau\cdot\shrink(\tau^{-1}u +A^Ty^{k})]_{\mathcal{X}}^+ - [\tau\cdot\shrink(\tau^{-1}u +A^T\hat{y})]^+_{\mathcal{X}} \rangle \\
= & \langle A^Ty^k-A^T\hat{y}, \prox_{\mathcal{X}_\tau(\cdot)}(u+\tau\cdot A^Ty^{k}) - \prox_{\mathcal{X}_\tau(\cdot)}(u+\tau\cdot A^T\hat{y}) \rangle  \\
\geq&\tau^{-1}\cdot  \|\prox_{\mathcal{X}_\tau(\cdot)}(u+\tau\cdot A^Ty^{k}) - \prox_{\mathcal{X}_\tau(\cdot)}(u+\tau\cdot A^T\hat{y}) \rangle\|_2^2\\
=&\tau^{-1} \cdot \| x^{k+1}- x^*\|_2^2
\end{align}
\end{subequations}
Using this inequality, we have
 \begin{subequations}
\begin{align}
\|y^{k+1}-\hat{y}\|_2^2=&\|y^k-\hat{y} +h (b-Ax^{k+1})\|_2^2 \\
=&\|y^k-\hat{y} +h (Ax^*-Ax^{k+1})\|_2^2 \\
=& \|y^k-\hat{y}\|_2^2 -2h \langle  A^Ty^k- A^T\hat{y}, x^{k+1}- x^*\rangle + h^2\|Ax^*-Ax^{k+1}\|_2^2 \\
\leq & \|y^k-\hat{y}\|_2^2-2h\tau^{-1}\| x^{k+1}- x^*\|_2^2+h^2\|A\|^2\| x^{k+1}- x^*\|_2^2\\
=& \|y^k-\hat{y}\|_2^2-h(2\tau^{-1}-h\|A\|^2)\| x^{k+1}- x^*\|_2^2.
\end{align}
\end{subequations}
Therefore, under the assumption $0<h<\frac{2}{\tau\|A\|^2}$ we can make the following claims:

\textbf{claim 1:} $\|y^{k+1}-\hat{y}\|_2$ is monotonically nonincreasing in $k$ and thus converges to a limit;

\textbf{claim 2:} $\|x^{k+1}-x^*\|_2 $ converges to 0 as $k$ tends to $+\infty$, i.e., $\lim_{k\rightarrow +\infty} x^{k+1}=x^*$.

From claim 1, it follows that $\{y^k\}$ is bounded and thus has a converging subsequence $y^{k_i}$. Let $\bar{y}=\lim_{i\to\infty}y^{k_i}$. By the Lipschitz continuity of proximal point operator in Lemma \ref{lem1} and Corollary \ref{proshr}, we have
 \begin{subequations}
\begin{align}\nonumber
x^*&=\lim_{i\rightarrow \infty} x^{k_i+1}=\lim_{i\rightarrow \infty} [\tau \cdot \shrink(\tau^{-1}u +A^Ty^{k_i+1})]_\mathcal{X}^+\\ \nonumber
&=\lim_{i\rightarrow \infty} \prox_{\mathcal{X}_\tau(\cdot)}(u+\tau\cdot A^Ty^{k_i+1})= \prox_{\mathcal{X}_\tau(\cdot)}(u+\tau\cdot A^T\bar{y})=  [\tau \cdot \shrink(\tau^{-1}u +A^T\bar{y})]_\mathcal{X}^+,\nonumber
\end{align}
\end{subequations}
so $\bar{y}\in \mathcal{Y}$ by Lemma \ref{dset}. Recall $\hat{y}\in\mathcal{Y}$ is arbitrary. Hence, claim 1 holds for $\hat{y}=\bar{y}$. If $\{y^k\}$ had another limit point, then $\|y^{k+1}-\bar{y}\|_2$ would fail to be monotonic. So, $y^k$ converges to $\bar{y}\in\mathcal{Y}$ (in norm).
\end{proof}


%
\section{Numerical experiment}
In the section, we do sparse recovery experiments to demonstrate that adding box constraints can help improve recovery of sparse signals considerably. It was shown in \cite{ly} when the augmented parameter $\tau\geq 10\|x\|_\infty$, the augmented $\ell_1$-norm model \eqref{Augo} is equivalent to classical basis pursuit \eqref{BPo} if the sensing matrix $A$ satisfies certain properties such as null-space property, or restricted isometry property. So we only test models \eqref{Augo} and \eqref{Aug} to observe possible advantages of adding box constraints. In the test, model \eqref{Augo} was solved by the LBreg algorithm and model \eqref{Aug} by the ProShrink algorithm.

We used 100 random pairs $(A, x)$ with matrices $A$ of size $200\times 400$ and vectors $x$ with 400 entries, out of which $s$ were nonzero entries set to $\pm 1$ uniformly randomly for $s=1, 2, 3, \cdots, 80$. Each entry of the sensing  matrix $A$
was sampled independently from the standard Gaussian distribution. Thus, $b=Ax$ are given vectors. A relative error of $10^{-12}$ was considered as an exact recovery; the relative error is defined as $\frac{x-x_o}{x}$ where $x_o$ is finally generated by the LBreg or the ProShrink algorithms. The box-constrained set $\mathcal{X}$ for the ProShrink algorithm was set as $[-1, 1]^{400}$.

We plot the exact recovery rate via sparsity levels in Figure \ref{fig4} from which we see that ProShrink performs remarkably better than LBreg as the sparse level increases. More precisely, when the sparse level is low, both LBreg and ProShrink can well recover sparse signals; but when the sparse level becomes high, the recovery rate by LBreg is worse than that by ProShrink that indicates adding box constraints to the augmented $\ell_1$-norm model \eqref{Augo} indeed improves the recovery rate.

\begin{figure}[ht]
\centering
    \includegraphics[scale =0.5] {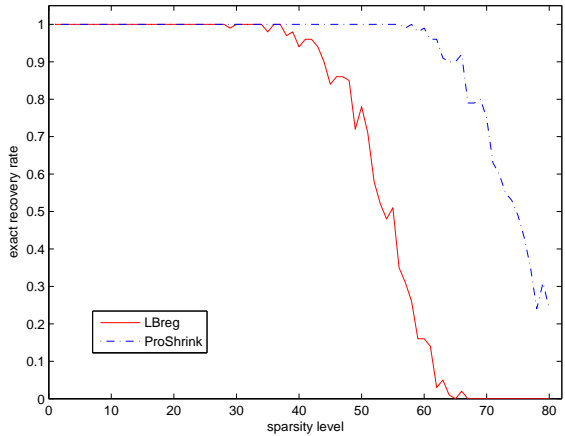}
\caption[Optional caption for list of figures]{Comparison of augmented $\ell_1$ norm models with or without box constraints for sparse recovery (correspond to the ProShrink and the LBreg algorithms separately). }
\label{fig4}
\end{figure}

\section{Conclusion}
In this paper, we proposed the projected shrinkage algorithm for boxed-constrained $\ell_1$-minimization. The most important factor in our study should be the deduction of formulation \eqref{key} that establishes the relationship between projected shrinkage operator and proximal point operator. Numerically, we demonstrated that adding box constraints to classical $\ell_1$-minimization can obtain better performance. However, giving theoretical explanation for this phenomenon is open. We leave it for future work.

\section*{Acknowledgements}
We would like to thank Professor  Wotao Yin (UCLA) for his comments and suggestion on numerical verification and Professor Jian-Feng Cai (Iowa U) for his insight of the projected shrinkage operator.

\small{

}

\end{document}